\documentclass{gIPE2e}
\usepackage{amsthm,amsfonts,amssymb,amsmath,amsbsy}
\usepackage{subfigure} 
\usepackage{graphicx} 
\usepackage{color}		
\usepackage{epsfig}

\newcommand{\intL}{\int\limits}
\newcommand{\half}{^\infty_0 }
\newcommand{\intR}{\int\limits_{\mathbb{R}} }
\newcommand{\intRR}{\int\limits_{\mathbb{R}^2} }
\newcommand{\RR}{\mathbb{R}}

\newcommand{\vx}{\vec{x}}

\newcommand{\uualpha}{\hat{\alpha}} 
\newcommand{\uutheta}{\hat{\theta}} 
\newcommand{\uuomega}{\hat{\omega}} 

\newtheorem{thm}{Theorem}
\newtheorem{defi}[thm]{Definition}

\newtheorem{cor}[thm]{Corollary}
\newtheorem{lem}[thm]{Lemma}

\title{A Radon-type transform arising in Photoacoustic Tomography with circular detectors: spherical geometry}

\author{Yulia Hristova${}^a$ , Sunghwan Moon${}^b$ $^{\ast}$\thanks{$^\ast$Corresponding author. Email: shmoon@unist.ac.kr
\vspace{6pt}}, and Dustin Steinhauer${}^c$ 
\\\vspace{6pt}  $^{a}${\em{Department of Mathematics and Statistics,\\
 University of Michigan-Dearborn, Dearborn, MI, 48128}};\\
$^{b}${\em{Department of Mathematical Sciences,\\
Ulsan National Institute of Science and Technology,
Ulsan 689-798, Republic of Korea}};\\
$^{c}${\em{Department of Mathematics,\\ Texas A\&M University, 
College Station, TX 77840-3368}} }

\begin{document}
\maketitle
 \begin{abstract}
 This paper is devoted to a Radon-type transform arising in a version of Photoacoustic Tomography that uses integrating circular detectors. 
We show that the transform can be decomposed into the spherical Radon transform and the Funk-Minkowski transform.

An inversion formula and a range description are obtained by combining existing inversion formulas and range descriptions for the spherical Radon transform and the Funk-Minkowski transform.
Numerical simulations are performed to demonstrate our proposed algorithm.
 \end{abstract}

\begin{classcode}44A12;65R10;92C55;35L05\end{classcode}
\section{Introduction}
Hybrid biomedical imaging modalities have been an active area of research lately, combining different physical signals in order to utilize their advantages to enhance images.
Photoacoustic Tomography (PAT), which incorporates ultrasound and optical or radio-frequency electomagnetic waves, is one of the most successful examples of such a combination.
While pure ultrasound imaging typically leads to high resolution images, the contrast between cancerous and healthy tissue is rather low.
 On the other hand, optical or radio-frequency electromagnetic imaging can provide a significant contrast; however, it has low resolution. 
The photoacoustic effect, discovered by A.G. Bell~\cite{bell80}, allows one to take advantage of the strengths of the pure optical and ultrasound imaging, and it forms the basis of PAT. 

In PAT, an object of interest is irradiated by pulsed non-ionizing electromagnetic energy. Due to the photoacoustic effect, an acoustic wave dependent on the electromagnetic absorption properties of the object is generated~\cite{kuchmentk08,xuw06}. 
This wave is then measured by ultrasound detectors placed outside the object. 
The internal photoacoutic sources are reconstructed from the measurements to produce a 3-dimensional image (tomogram). 
Irradiated cancerous cells, in particular, are displayed with high contrast. This is due to the fact that such cells absorb several times more electromagnetic energy than healthy tissue, and electromagnetic deposition is proportional to the strength of the generated acoustic wave. Since the initially generated acoustic wave contains diagnostic information, one of the classical mathematical problems of PAT is the recovery of this initial acoustic pressure field from ultrasound measurements made outside the object. 

The initial approach to detecting ultrasound signals in PAT has been to use small piezoelectric transducers. As these transducers mimic point-like measurements, reconstruction algorithms produce images with a spatial resolution limited by the size of the transducers. 
Another drawback of such detectors is the difficulty manufacturing small transducers with high sensitivity. 
To overcome these deficiencies, various other types of acoustic detectors have been introduced, e.g. linear, planar, cylindrical and circular detectors~~\cite{burgholzerbmghp07,grattpnp11,haltmeier09,haltmeiersbp04,mooncrt13,zangerls09}. These detectors are modeled as measuring the integrals of the pressure over the shape of the detector.

There are works discussing PAT with circular integrating detectors~\cite{moontat14,smoon,zangerls10,zangerlsh09,zangerls09}.
Some works~\cite{moontat14,smoon,zangerlsh09,zangerls09} have considered PAT with circular integrating detectors on cylindrical geometry: a stack of parallel circular detectors whose centers lie on a cylinder. 
It was shown that the data from PAT with circular detectors is the solution of a certain initial value problem, and this fact was used to reduce the problem of recovery of the initial pressure to that of inversion of the circular Radon transform in~\cite{zangerls10,zangerlsh09,zangerls09}.
The cases where the centers of the detector circles are located on a plane or a sphere have been discussed in~\cite{moontat14}.
In the present work we study PAT with circular integrating detectors in a spherical geometry, which can be useful in investigating small objects~\cite{zangerls10}.
We note that a type of spherical geometry was also discussed in~\cite{zangerls10}. However, the geometry considered here is different, as only detectors with the same radius are required, while various sizes of circular detectors are needed in~\cite{zangerls10}.
Also, our approach is to define a Radon-type transform arising in this version of PAT, and we show that this transform can be decomposed into the spherical Radon transform and the Minkowski-Funk transform, both of which are well studied.
  
This paper is organized as follows.
Section~\ref{defiandwork} is devoted to the construction of a Radon-type transform arising in PAT with circular integrating detectors. 
In Section~\ref{recon}, we show that the Radon-type transform is the composition of the Minkowski-Funk transform and the spherical Radon transform with centers on a sphere, and we present the inversion formula using this fact.
In Section~\ref{range}, we describe the range of this Radon-type transform, and we give necessary and sufficient conditions for a function to arise as data measured by the circular detectors in the spherical geometry.
Section~\ref{S:numerical} is concerned with numerical implementation.
Section~\ref{S:conclusions} contains some final remarks.
The appendices provide proofs of some technical statements.
\section{Photoacoustic tomography with circular integrating detectors}\label{defiandwork}
In PAT, the acoustic pressure $p(\vx,t)$ satisfies the following initial value problem:
\begin{equation}\label{eq:pdeofpat}
\begin{array}{lll}
 &\partial^2_t\, p(\vx,t)=\Delta_{\vx}\, p(\vx,t)\qquad&(\vx,t)\in\RR^3\times(0,\infty),\\
&p(\vx,0)=f(\vx)\qquad&\vx=(x_1,x_2,x_3)\in\RR^3,\\
&\partial_t \,p(\vx,0)=0 \qquad&\vx\in\RR^3.
\end{array}
\end{equation}
Here we assume that the sound speed is equal to $1$ everywhere, including the interior of the object. The goal of PAT is to recover the initial pressure $f$ from measurements of $p$ outside the support of $f$.

Throughout this article, it is assumed that the initial pressure field $f$ is smooth and supported in $B(0,r_{det})$ - the ball in $\RR^3$ with radius $r_{det}$ centered at the origin.
We also assume that the acoustic signals are measured by a set of circular detectors centered at the origin with radius $r_{det}$ - that is, the detector circles are great circles on the sphere $\partial B(0,r_{det})$ (see Figure~\ref{fig:tatsphere}). 
\begin{figure}%
\begin{center}
\includegraphics[width=0.45\textwidth]{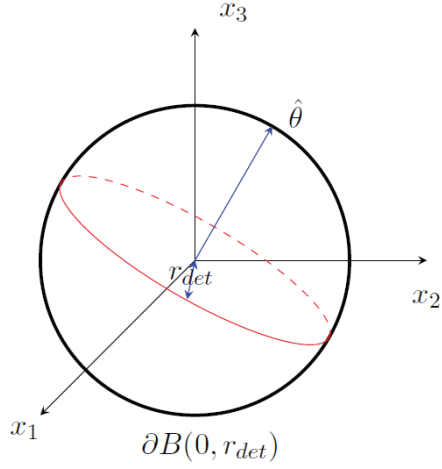}%
\label{fig:tatsphere}
\caption{The great circle on the sphere $\partial B(0,r_{det})$ of radius $r_{det}$ in the plane perpendicular to $\uutheta$.}
\end{center}
\end{figure}
%

The data $P(\uutheta,t)$, for $(\uutheta,t)\in S^2\times(0,\infty)$ measured by the circular detector lying in the plane  perpendicular to $\uutheta$ (see Figure~\ref{fig:tatsphere}) can be expressed as
$$
P(\uutheta,t)=\frac{1}{2\pi}\intL_{\uutheta\cdot\uualpha=0}p(r_{det}\uualpha,t)dS(\uualpha),
$$
where $dS$ is the measure on the unit circle $S^1$. 

Everywhere in the text we will use the hat notation to indicate a unit vector in 2D or 3D, and $dS$ to denote the measure either on the circle or on the sphere. The dimension will be clear from the context.

Using Kirchhoff's formula for the solution of ~\eqref{eq:pdeofpat} (e.g. \cite{EvansBook}):
$$
p(\vx,t)=\partial_t\left(\frac{1}{4\pi t}\intL_{\partial B (\vx,t)}f(\hat\beta)dS(\hat\beta)\right),
$$
we can represent the measurements $P(\uutheta,t)$ as follows:
\begin{equation}\label{eq:definitionofP}
\begin{array}{ll}
P(\uutheta,t)&\displaystyle=\frac{1}{2\pi}\intL_{\uutheta\cdot\uualpha=0}\partial_t\left(\frac{1}{4\pi t}\intL_{\partial B (r_{det}\uualpha,t)}f(\hat\beta)dS(\hat\beta) \right)d S(\uualpha)\\
&\displaystyle=
\frac{1}{8\pi^2}\partial_t\left(t\intL_{\uutheta\cdot\uualpha=0} \intL_{S^2}
f(r_{det}\uualpha+t\hat{\beta}) \,dS(\hat{\beta}) dS(\uualpha)\right),
\end{array}
\end{equation}
where $S^2$ is the unit sphere in $\RR^3$.
Let us define the new Radon-type transform 
$$
\displaystyle \mathcal{R}_Pf(\uutheta,t)=\intL_{\uutheta\cdot\uualpha=0} \intL_{S^2}
f(r_{det}\uualpha+t\hat{\beta}) dS(\hat{\beta}) dS(\uualpha).
$$
Then, \eqref{eq:definitionofP} reads as
\begin{equation}
P(\uutheta,t) = \frac{1}{8 \pi^2}\partial_t \left( t\mathcal{R}_P f(\uutheta,t)\right).
\label{eq:PandR}
\end{equation}

If $f$ is odd, i.e, $f(\vx)=-f(-\vx)$, then $\mathcal{R}_Pf$ is equal to zero. We thus assume that $f$ is even. This entails no loss of generality, as any function compactly supported in the upper hemisphere of  $B (0,r_{det})$ can be extended to an even function with support in $B (0,r_{det})$.  
In what follows, $\mathcal{R}_P$ is seen to be the composition of the spherical Radon transform with centers on $\partial B (0,r_{det})$ and the Minkowski-Funk transform. 
This observation will allow us to recover $f$ from the measurements $P$ as well as to describe necessary and sufficient conditions for a given function to arise as data measured by the circular integrating detectors.

\section{Reconstruction}\label{recon}
The following definitions will be used in the rest of this paper. 

\begin{defi}\indent
 The Minkowski-Funk transform $F$ maps a locally integrable function $\phi$ defined on $S^2\times[0,\infty)$ such that  $\phi(\uualpha,t)=\phi(-\uualpha,t)$ into
$$
F\phi(\uutheta,t)=\intL_{\uutheta\cdot\uualpha=0}\phi(\uualpha,t)dS(\uualpha), \qquad\mbox{ for  }(\uutheta,t)\in S^2\times[0,\infty).
$$
\end{defi}
\begin{defi}\indent
 The spherical Radon transform $R_S$ maps a locally integrable function $f$ on $\RR^3$ into
$$
R_Sf(\uualpha,t)=\intL_{S^2}f(r_{det}\uualpha+t\hat{\beta})dS(\hat{\beta}),\qquad\mbox{ for  } (\uualpha,t)\in S^2\times[0,\infty).
$$
\end{defi}
Now it is easily seen that the following representation of $\mathcal{R}_P f$ holds.
\begin{thm}\label{thm:R_P}
For any $f\in C^\infty(\RR^3)$ compactly supported in $B(0,r_{det})$ with $f(\vec x)=f(-\vec x)$, we have $\mathcal{R}_Pf(\uutheta,t)=F(R_Sf)(\uutheta,t)$.
\end{thm}
Thus, in order to reconstruct the initial pressure $f$, we will employ known inversion formulas for the Minkowski-Funk and the spherical Radon transforms. A number of inversion formulas for these transforms have been derived, e.g. some of those for the spherical Radon transform can be found in~\cite{finchr09,finchpr04,kunyansky07,kunyansky071,Nguyen_FamilyInversionTAT} and those for the Minkowski-Funk transform are studied in~\cite{gelfandgg03,gindikinrs93,helgason99radon,nattererw01}.  
We choose the inversion formula for the spherical Radon transform given in~\cite{finchr09,finchpr04} and present the relation between the Minkowski-Funk transform and the regular Radon transform in the following theorem:
\begin{thm}\label{thm:funkradon}
Let $\phi$ be a continuous and even function on $S^2$, i.e., $\phi(\uutheta)=\phi(-\uutheta)$.
If 
$$
F\phi(\uutheta)=\intL_{\uutheta\cdot\uualpha=0}\phi(\uualpha)dS(\uualpha), \qquad\mbox{ for  } \uutheta\in S^2,
$$
then, if $\uutheta\neq(0,0,1)$ and $\uutheta\neq(0,0,-1)$, the following relation holds,
\begin{equation*}\label{eq:relationbetweenradon}
F\phi(\uutheta)=\frac{1}{|\theta'|}R\Phi\left(\frac{\theta'}{|\theta'|},-\frac{\theta_3}{|\theta'|}\right),
\end{equation*}
where $\uutheta=(\theta',\theta_3)=(\theta_1,\theta_2,\theta_3)\in S^2, \uualpha = (\alpha',\alpha_3)\in S^2$, $\Phi(\alpha'/\alpha_3)=2\phi(\uualpha)\alpha_3^2$, and  
$$
R[\Phi(x_1,x_2)](\uuomega,s)=\intR \Phi(s\uuomega+\nu \uuomega^\perp)d\nu,\quad\mbox{ for  } (\uuomega,s)\in S^1\times\RR.
$$
\end{thm}
The proof of this Theorem is provided in Appendix A. 
Our main result is given in the theorem below.
\begin{thm}\label{thm:inversion}
Let $f\in C^\infty(\RR^3)$ be an even function with compact support in $B(0,r_{det})$.
Then, if $G(\uuomega,s,t)=(1+s^2)^{-\frac12}\mathcal R_pf\left(\uuomega/\sqrt{1+s^2},-s/\sqrt{1+s^2},t\right)$, the following relation holds:
\begin{equation}\label{eq:inversion}
\begin{split}
 f(\vx) = -\frac{r_{det}}{16\pi^2}\intL_{S^2}\left(\intL_{S^1}p.v. \intR  \frac{ \partial_s\partial_t t \partial_t t G(\uuomega,s,t)|_{t=|r_{det}\uualpha-\vx|}}
{\alpha_3\,\uuomega\cdot\alpha'-s\alpha_3^2}\; \frac{ds\, dS(\uuomega)}{|r_{det}\uualpha-\vx|}
\right)dS(\uualpha).
\end{split}
\end{equation}
\end{thm}
\begin{proof}

From Theorem~\ref{thm:funkradon}, we have
\begin{equation}
F\phi(\uutheta,t)=\frac{1}{|\theta'|} R[\Phi(x_1,x_2,t)] \left(\frac{\theta'}{|\theta'|},-\frac{\theta_3}{|\theta'|},t\right),
\label{eq:FunkRadon}
\end{equation}
where $\uutheta=(\theta',\theta_3)$, $\uualpha = (\alpha',\alpha_3)\in S^2$, $\Phi(\alpha'/\alpha_3,t)=2\phi(\uualpha,t)\alpha_3^2$ and $R$ is the 2D  Radon transform of $\Phi$ with respect to the first two variables. That is,
$$
R[\Phi(x_1,x_2,t)](\uuomega,s,t)=\intR \Phi(s\uuomega+\nu \uuomega^\perp,t)d\nu,\quad\mbox{ for  } (\uuomega,s)\in S^1\times\RR.
$$
After changing variables in (\ref{eq:FunkRadon}) we obtain the relation 
\begin{equation}
R[\Phi(x_1,x_2,t)] (\uuomega,s,t)=\frac{1}{\sqrt{1+s^2}}F\phi\left(\frac{\uuomega}{\sqrt{1+s^2}},\frac{-s}{\sqrt{1+s^2}},t\right).
\label{eq:FunkRadon2}
\end{equation}
In order to recover $\Phi$, we use the filtered backprojection formula for the inversion of the 2D Radon transform
\footnote{For completeness, we verify the application of the filtered backprojection formula to the function $\Phi$
in Appendix B.} 
\begin{equation}
\label{eq:FBP}
\begin{split}
\Phi(x_1,x_2,t) =& \frac{1}{4\pi} (R^\# H \partial_s R\Phi)(x_1,x_2,t)\\ 
=& \frac{1}{4\pi^2} \int\limits_{S^1} \left.\left( p.v. \intR \frac{\partial_{s'} (R\Phi) (\uuomega,s',t)}{s-s'} \,ds'\right)\right|_{s = \vx\cdot\uuomega}\,dS(\uuomega),
\end{split}
\end{equation}
where $\displaystyle R^\# u(x_1,x_2):=\int_{S^1}u(\uuomega,(x_1,x_2)\cdot\uuomega)\,dS(\uuomega)$ is the backprojection operator and $\displaystyle H u(s) := p.v.\frac{1}{\pi}\int_\RR \frac{u(s')}{s-s'}\,ds'$ is the Hilbert transform.\\

Applying (\ref{eq:FBP}) to (\ref{eq:FunkRadon2}) leads to the inversion formula for the Minkowski-Funk transform:
\begin{equation}
\label{eq:inversionofmf}
\begin{split}
\phi(\uualpha,t)=& 2^{-1}\alpha_3^{-2}\Phi(\alpha'/\alpha_3,t)\\
=& \frac{1}{8\pi \alpha_3^2}\,R^\#\left( H \partial_s \left[\frac1{\sqrt{1+s^2}} F\phi\left( \frac{\uuomega}{\sqrt{1+s^2}},\frac{-s}{\sqrt{1+s^2}},t\right)\right]\right) \left(\frac{\alpha'}{\alpha_3},t\right)
\end{split}
\end{equation}
Now let $\phi(\uualpha,t) = (R_S f)(\uualpha,t)$. Theorem \ref{thm:R_P} implies that $F\phi = \mathcal{R}_P f$. Hence, we obtain the following expression for the spherical Radon transform of $f$:
\begin{equation}
\label{eq:R_S f}
\begin{split}
R_S f(\uualpha,t)=& \frac{1}{8\pi \alpha_3^2}\,R^\#\left( H \partial_s \left[\frac1{\sqrt{1+s^2}}  (\mathcal{R}_P f)\left( \frac{\uuomega}{\sqrt{1+s^2}},\frac{-s}{\sqrt{1+s^2}},t\right)\right]\right) 
\left(\frac{\alpha'}{\alpha_3},t\right)\\
=& \frac{1}{8\pi \alpha_3^2}\,R^\# H \partial_s \left[G(\uuomega,s,t)\right]
\left(\frac{\alpha'}{\alpha_3},t\right).
\end{split}
\end{equation}

In order to recover $f$, we apply the inversion formula for $R_S$ given in \cite{finchpr04}:
\begin{equation}\label{eq:inversionofspherical}
f(\vx)=-\frac{r_{det}}{8\pi^2}\intL_{S^2}\left.\frac{\partial_tt\partial_ttR_Sf(\uualpha,t)}{t}\right|_{t=|r_{det}\uualpha-\vx|}dS(\uualpha).
\end{equation}
Combining two equations~\eqref{eq:R_S f} and~\eqref{eq:inversionofspherical} gives us 
\begin{equation*}
\begin{split}
 f(\vx) = -\frac{r_{det}}{64\pi^3}\int\limits_{S^2}\frac{1}{\alpha_3^2} R^\#\left[ H \partial_s \;\partial_t t \partial_t t G(\uuomega,s,t)\right]
 \left(\frac{\alpha'}{\alpha_3},t\right)\bigg|_{t=|r_{det}\uualpha-\vx|} \frac{dS(\uualpha)}{|r_{det}\uualpha - \vx|},
\end{split}
\end{equation*}
which is equivalent to (\ref{eq:inversion}).
\end{proof}
\begin{cor}
Let $f\in C^\infty(\RR^3)$ be an even function with compact support in $B(0,r_{det})$.
Then, if $g(\uuomega,s,t)=(1+s^2)^{-\frac12}P\left(\uuomega/\sqrt{1+s^2},-s/\sqrt{1+s^2},t\right)$, the following relation holds
\[
f(\vx)=-\frac{r_{det}}{8\pi^2}\intL_{S^2}\left(\intL_{S^1}p.v. \intR  \frac{ \partial_s \partial_t t g(\uuomega,s,t)|_{t=|r_{det}\uualpha-\vx|}}
{\alpha_3\,\uuomega\cdot\alpha'-s\alpha_3^2}\; \frac{ds\, dS(\uuomega)}{|r_{det}\uualpha-\vx|}
\right)dS(\uualpha).
\]
\end{cor}
Indeed, since $P(\uutheta,t)=(8\pi^2)^{-1}\partial_t t\mathcal R_P f(\uutheta,t)$ (see~\eqref{eq:PandR}) we can write
\begin{equation*}\label{eq:gandG}
g(\uuomega,s,t) =\frac1{\sqrt{1+s^2}} P\left(\frac{\uuomega}{\sqrt{1+s^2}},\frac{-s}{\sqrt{1+s^2}},t\right)=\frac{\partial_tt G(\uuomega,s,t)}{8\pi^2\sqrt{1+s^2}},
\end{equation*}
where $G(\uuomega,s,t)$ is the function defined in the preceding theorem. Then the result follows immediately from Theorem~\ref{thm:inversion}.

\section{Range description}\label{range}
In this section, we describe the range of $\mathcal R_P$ and necessary and sufficient conditions for a function $P$ to arise as data measured by the circular detectors.

As mentioned before, the Minkowski-Funk transform and the spherical Radon transform are well studied, so their range descriptions have already been discussed (see \cite{agranovskykq07,agranovskyfk09,agranovskyn10,gardner95,groemer96}). 
In~\cite{gardner95,groemer96}, it was proved that the Minkowski-Funk transform is a (continuous) bijection from $C^\infty_e(S^{n-1})$ to itself, where $C^\infty_e(S^{n-1})=\{\phi\in C^\infty(S^{n-1}):\phi(\hat \theta)=\phi(-\hat\theta)\}$.
In~\cite{agranovskykq07}, M. Agranovsky, P. Kuchment, and E. T. Quinto provided four different types of complete range descriptions for the spherical mean transform.
In particular, one of the range descriptions states that necessary and sufficient conditions for the function $g\in C^\infty_c(S^2\times[0,2r_{det}])$ to be representable as $g= R_Sf$ for some $f\in C^\infty_c(B(0,r_{det}))$, is that for any integers $m,l$, the $(m,l)$-th order spherical harmonic term $(\mathbf H_{\frac n2-1}g)_{m,l}(\lambda)$ of $\mathbf H_{\frac n2-1}g(\hat\theta,\lambda)$ vanishes at non-zeros of the Bessel function $J_{m+n/2-1}(\lambda)$, where 
$$
\mathbf H_{\frac n2-1}g(\hat\theta,\lambda)=\intL\half g(\hat\theta,t)j_{\frac n2-1}(\lambda t)t^{n-1}dt.
$$
Here $j_n(t)$ is the spherical Bessel function:
$$
j_n(t)=\frac{2^n\Gamma(n+1)J_n(t)}{t^n},
$$
where $J_n(t)$ is the Bessel function of the first kind of order $n$ and $\Gamma(t)$ is the Gamma function.

Combining these two range descriptions, we have the following range description for $\mathcal R_P$:
\begin{thm}\label{thm:rangeofrp}
Let $g\in C^\infty_c(S^2\times[0,2r_{det}])$. 
The following conditions are necessary and sufficient for the function $g$ to be representable as $g=\mathcal R_Pf$ for some $f\in C^\infty_c(B(0,r_{det}))$:
\begin{itemize}
\item $g$ is even in $\hat \theta\in S^2$, i.e., $g(\hat\theta,t)=g(-\hat\theta,t)$.
\item For any positive integer $n$, let 
$$
\mathbf H_{\frac n2-1}g(\hat\theta,\lambda)=\intL\half g(\hat\theta,t)j_{\frac n2-1}(\lambda t)t^{n-1}dt. 
$$
Then for any integers $m,l$, the $(m,l)$-th order spherical harmonic term $(\mathbf H_{\frac n2-1}g)_{m,l}(\lambda)$ of $\mathbf H_{\frac n2-1}g(\hat\theta,\lambda)$ vanishes at non-zeros of the Bessel function $J_{m+n/2-1}(\lambda)$.
\end{itemize}
\end{thm}
\begin{proof}
It is enough to show that the $(m,l)$-th order spherical harmonic term $(\mathbf H_{\frac n2-1}g)_{m,l}(\lambda)$ of $\mathbf H_{\frac n2-1}g(\hat\theta,\lambda)$ vanishes at non-zeros of the Bessel function $J_{m+n/2-1}(\lambda)$ if and only if
the $(m,l)$-th order spherical harmonic term $(\mathbf H_{\frac n2-1}[F^{-1}g])_{m,l}(\lambda)$ of $\mathbf H_{\frac n2-1}[F^{-1}g](\hat\theta,\lambda)$ vanishes at non-zeros of the Bessel function $J_{m+n/2-1}(\lambda)$, where $F^{-1}$ is the inversion of the Minkowski-Funk transform.

In~\cite{helgason99radon}, the inversion formula for the Funk-Minkowski transform is presented as follows:
$$
\phi(\uualpha)=\frac{1}{4\pi^2}\left.\left\{\frac{d}{du}\intL^u_0\intL_{S^2} F\phi(\uutheta)\delta(\uutheta\cdot\uualpha-\sqrt{1-v^2})\frac{dS(\uutheta)dv}{(u^2-v^2)^{1/2}}\right\}\right|_{u=1},
$$
where $\delta$ is the Dirac delta function.
Therefore, $F^{-1}g(\uualpha,t)$ can be represented as
$$
\frac{1}{4\pi^2}\left.\left\{\frac{d}{du}\intL^u_0\intL_{S^2} g(\uutheta,t)\delta(\uutheta\cdot\uualpha-\sqrt{1-v^2})\frac{dS(\uutheta)dv}{(u^2-v^2)^{1/2}}\right\}\right|_{u=1}.
$$
Consider the $(m,l)$-th order spherical harmonic term $(\mathbf H_{\frac n2-1}[F^{-1}g])_{m,l}(\lambda)$:
\begin{equation}\label{eq:mthorderspherical}
\begin{split}
&(\mathbf H_{\frac n2-1}[F^{-1}g])_{m,l}(\lambda)=\intL_{S^2}\mathbf H_{\frac n2-1}[F^{-1}g](\uualpha,\lambda)Y_{l}^m(\uualpha)dS(\uualpha)\\
&=\frac{1}{4\pi^2}\intL_{S^2}\left.\left\{\frac{d}{du}\intL^u_0\intL_{S^2} \mathbf H_{\frac n2-1}g(\uutheta,\lambda)\delta(\uutheta\cdot\uualpha-\sqrt{1-v^2})\frac{dS(\uutheta)dv}{(u^2-v^2)^{1/2}}\right\}\right|_{u=1}Y_{l}^m(\uualpha)dS(\uualpha)\\
&=\frac{1}{4\pi^2}\left.\left\{\frac{d}{du}\intL^u_0\intL_{S^2} \mathbf H_{\frac n2-1}g(\uutheta,\lambda)\intL_{S^2}\delta(\uutheta\cdot\uualpha-\sqrt{1-v^2})Y^m_{l}(\uualpha)dS(\uualpha)\frac{dS(\uutheta)dv}{(u^2-v^2)^{1/2}}\right\}\right|_{u=1},
\end{split}
\end{equation}
where in the last line, we used the Fubini-Tonelli Theorem.

To compute the inner integral with respect to $\uualpha$ of \eqref{eq:mthorderspherical}, we need the Funk-Hecke theorem~\cite{daix13,natterer01,seeley66}: if $\int^1_{-1}|h(t)|dt<\infty$,
\begin{equation*}
\begin{split}
&\intL_{S^{n-1}}h(\uutheta\cdot\uuomega)Y^m_{l}(\uuomega)dS(\uuomega)=c(n,m)Y^m_{l}(\uutheta),\\
&c(n,m)=|S^{n-2}|\intL^1_{-1}h(t)C^{(n-2)/2}_m(t)(1-t^2)^{(n-3)/2}dt.
\end{split}
\end{equation*}
Here $C^\nu_m(t),\nu>-1/2$ is the Gegenbauer polynomial of degree $m$ and 
$C^\frac12_m(t)=P_m(t)$ is the Legendre polynomial of degree $m$.
Applying the Funk-Hecke theorem to the inner integral of \eqref{eq:mthorderspherical}, the $(m,l)$-th order spherical harmonic term $(\mathbf H_{\frac n2-1}[F^{-1}g])_{m,l}(\lambda)$ is equal to
\begin{equation*}
\begin{split}
&\frac{1}{2\pi}\left.\left\{\frac{d}{du}\intL^u_0\intL_{S^2} \mathbf H_{\frac n2-1}g(\uutheta,\lambda)Y^m_{l}(\uutheta)dS(\uutheta)\;\frac{P_m(\sqrt{1-v^2})}{(u^2-v^2)^{1/2}}dv\right\}\right|_{u=1}.
\end{split}
\end{equation*}
Thus, we have
\begin{equation}\label{eq:mltheorderspherical}
(\mathbf H_{\frac n2-1}[F^{-1}g])_{m,l}(\lambda)=\frac{(\mathbf H_{\frac n2-1}g)_{m,l}(\lambda)}{2\pi}\left.\left\{\frac{d}{du}\intL^u_0\frac{P_m(\sqrt{1-v^2})}{(u^2-v^2)^{1/2}}dv\right\}\right|_{u=1}.
\end{equation}
Since for any integer $m$, there exists a positive number $\epsilon$ such that $P_m(\sqrt{1-v^2})(u^2-v^2)^{-1/2}$ is not equal to zero on the interval $[1-\epsilon,1]$, $\int^u_0P_m(\sqrt{1-v^2})(u^2-v^2)^{-1/2}dv$ is not constant on $[1-\epsilon,1]$. 
Therefore, from \eqref{eq:mltheorderspherical}, we have $(\mathbf H_{\frac n2-1}g)_{m,l}(\lambda)=0$ if and only if $(\mathbf H_{\frac n2-1}[F^{-1}g])_{m,l}(\lambda)=0$. 
\end{proof}
Since $P(\uutheta,t)=(8\pi^2)^{-1}\partial_t t\mathcal R_P f(\uutheta,t)$, we have the following necessary and sufficient conditions for a function $P$ to be data measured by the circular detector from the initial pressure field.
\begin{thm}
Let $P\in C^\infty_c(S^2\times[0,2r_{det}])$. 
The following conditions are necessary and sufficient for the function $P$ to be data measured by the circular detectors from the initial pressure field $f\in C^\infty_c(B(0,r_{det}))$:
\begin{itemize}
\item $P$ is even in $\hat \theta\in S^2$, i.e., $P(\hat\theta,t)=P(-\hat\theta,t)$.
\item 
The integral of $P$ with respect to $t$ from $0$ to $2r_{det}$ is equal to zero.
\item For any positive integer $n$, let 
$$
\mathbf H_{\frac n2-1}P(\hat\theta,\lambda)=\intL\half \intL^t_0P(\hat\theta,\tau) j_{\frac n2-1}(\lambda t)t^{n-2}d\tau dt.
$$
Then for any integers $m,l$, the $(m,l)$-th order spherical harmonic term $(\mathbf H_{\frac n2-1}P)_{m,l}(\lambda)$ of $\mathbf H_{\frac n2-1}P(\hat\theta,\lambda)$ vanishes at non-zeros of the Bessel function $J_{m+n/2-1}(\lambda)$.
\end{itemize}
\end{thm}

\begin{proof}
Let us define a function $g$ on $S^2\times[0,2r_{det}]$ by
$$
g(\hat\theta,t)=\frac1t\intL^t_0P(\hat\theta,\tau)d\tau.
$$
Then $g$ is infinitely differentiable with respect to $t$ away from zero, $g(\hat\theta,0)=0$ is equivalent to $P(\hat\theta,0)=0$ by L'H\^opital's rule, and $g(\hat\theta,2r_{det})=0$ is equivalent to 
$$
\intL^{2r_{det}}_0P(\hat\theta,\tau)d\tau=0.
$$
It is sufficient to show that $g$ is infinitely differentiable and continuous at 0.
This will be accomplished in the following two lemmas, which are proved in Appendix A. 

\begin{lem}\label{lem:derivative}
Let $\phi\in C^\infty([0,\epsilon])$ for some $\epsilon>0$ satisfy $\phi(0)=0$, and let $n$ be a positive integer.
For any positive integer $k\leq n$, we have
$$
\intL^t_0\phi(\tau) \partial_t^{k}\left(\frac{(t-\tau)^{n+1}}{t}\right)d\tau\to0\qquad\mbox{as}\quad t\to 0.
$$
\end{lem}

Lemma \ref{lem:derivative} is used to prove:

\begin{lem}\label{lem:integral}
Let $\phi\in C^\infty([0,\epsilon])$ satisfy $\phi(0)=0$, and let $h(t)=\int_0^t\phi(\tau)d\tau/t$.
Then $h\in C^\infty([0,\epsilon])$ and $h(0)=0$.
\end{lem}

Theorem \ref{thm:rangeofrp} now immediately follows by applying Lemma \ref{lem:integral} to $g(\hat\theta,t)$ considered as a function of $t$ for fixed $\hat\theta$.
\end{proof}

\section{Numerical Experiments}\label{S:numerical}

In this section, we show the results of numerical experiments reconstructing the initial pressure field $f$ from the data $P(\uutheta,t) = \frac{1}{8 \pi^2}\partial_t \left( t\mathcal{R}_P f(\uutheta,t)\right)$.
We set $r_{det}=1$ and assume that $P(\uutheta,t)$ is given on a spherical grid in $\uutheta$ on the unit sphere and a uniform grid in $t$ on the interval $[0,2]$.

We follow a method of reconstruction using the formula in Theorem \ref{thm:inversion}:

\begin{enumerate}
 \item Compute $F^{-1}P(\cdot,t)$ for each fixed $t$ to get the values of $p(\uutheta,t)$ on $S^2\times[0,2]$
 \item Invert the wave operator to get $f(\vx)$
\end{enumerate}

Our phantoms are taken to be either the sum of multiples of characteristic
functions of balls
\begin{align}\label{eq:sharpphantom}
\chi_{\mathrm{sh}}(\vx)=\begin{cases}
             c &\mbox{if } |\vx-\vec{x_0}|<r\\
          0 &\mbox{if } |\vx-\vec{x_0}|\geq r
         \end{cases}
\end{align}
or the sum of $C^1$-functions of the form
\begin{equation}\label{eq:smoothphantom}
\chi_{\mathrm{sm}}(\vx)=c\varphi\left(|\vx-\vec{x_0}|^2\right),
\end{equation}
where $\varphi$ is defined by
\begin{align*}\label{eq:phantom_form}
 \varphi\left(s^2\right)=\begin{cases}
          1 &\mbox{if } |s|\leq r\\
          \displaystyle 1-3\left(\frac{s-r}{R-r}\right)^2+2\left(\frac{s-r}{R-r}\right)^3 &\mbox{if } r<|s|<R\\
          0 &\mbox{if } |s|\geq R
         \end{cases}
\end{align*}
and $0<r<R$. 
Our algorithm reconstructs the even part of the phantom, so we take phantoms supported in the upper half
of the unit ball where $x_3>0$.

The acoustic pressure generated by $\chi_{\mathrm{sh}}$ is
\begin{equation*}
 p(\vx,t)=\begin{cases}
         c \displaystyle \frac{|\vx-\vec{x_0}|-t}{2|\vx-\vec{x_0}|}&\mbox{if }|\vx-\vec{x_0}|-t<r\\
	  0&\mbox{if }|\vx-\vec{x_0}|-t\geq r\\
         \end{cases}
\end{equation*}
and the acoustic pressure generated by $\chi_{\mathrm{sm}}$ is given by
\begin{equation*}
 p(\vx,t)=c\frac{|\vx-\vec{x_0}|-t}{2|\vx-\vec{x_0}|}\varphi\left(\big( |\vx-\vec{x_0}|-t\big)^2\right)
\end{equation*}
(see \cite[Theorem 5.2]{haltmeierss05}).  In order to get the data $P(\uutheta,t)$ measured by our detectors,
we use a quadrature on each detector circle at each grid point in $t$, with the trapezoid rule and 100 points.

\begin{figure}[!htbp]
\begin{center}
\subfigure[Phantom]{\includegraphics[width=0.45\textwidth]{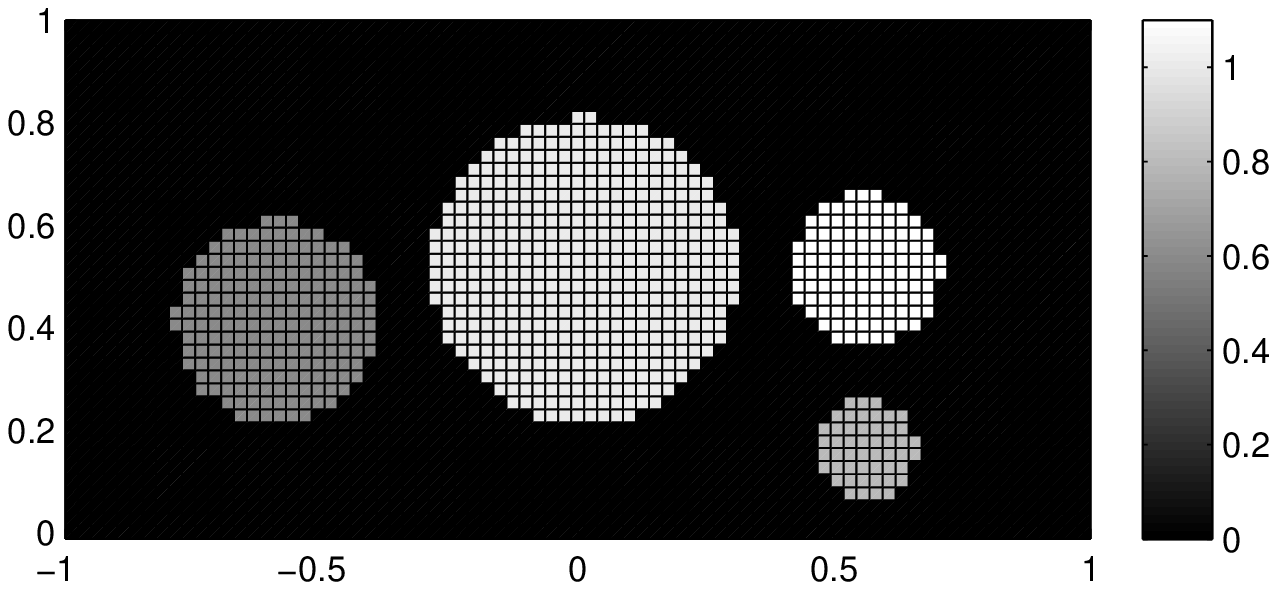}}     
\subfigure[Cross-section of the reconstruction in the $x_2x_3$-plane]{\includegraphics[width=0.45\textwidth]{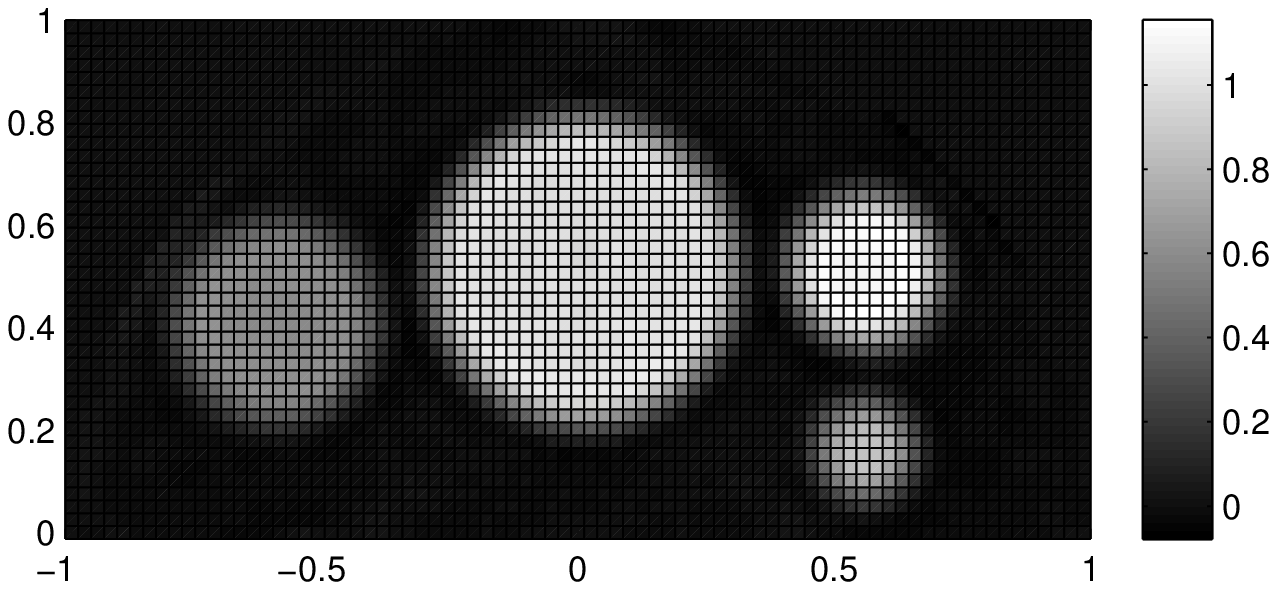}}\\
\subfigure[Slice of the reconstruction and phantom along the $x_3$-axis]{\includegraphics[width=0.45\textwidth]{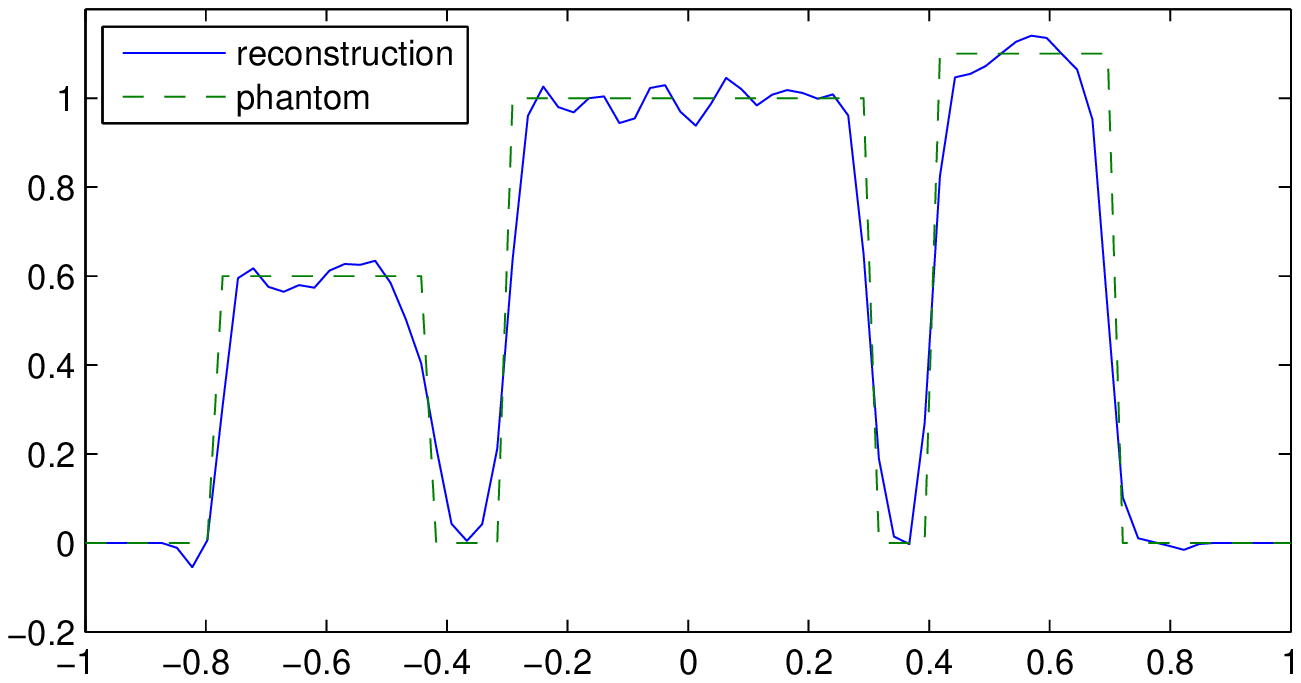}}
\subfigure[Slice of the reconstruction and phantom along the line $x_1=0,\;x_3=0.5$]{\includegraphics[width=0.45\textwidth]{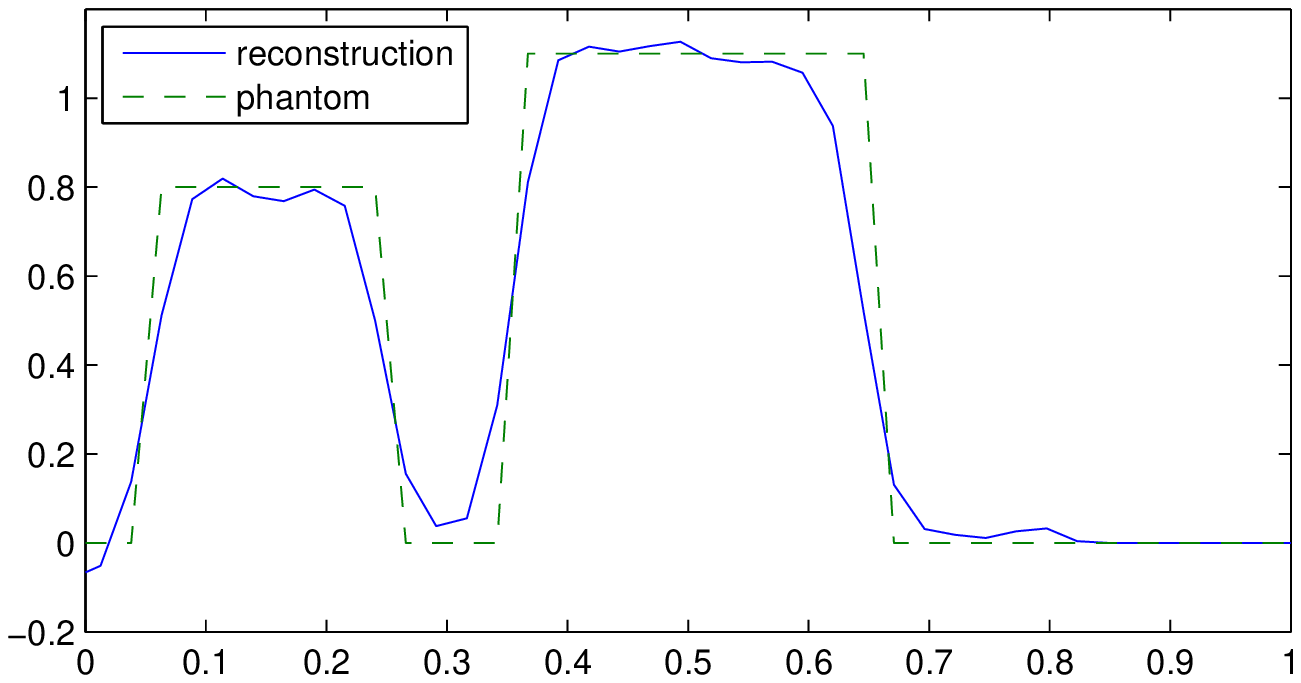}}
\end{center}
\caption{Reconstructions of a sharp phantom with centers $(0,-0.6,0.4)$, $(0,0,0.5)$, $(0,0.55,0.5)$,
and $(0,0.55,0.15)$, radii $0.2$, $0.3$, $0.15$, and $0.1$ and values 0.6, 1, 1.1, and 0.8, respectively}
\label{fig:4sharpballs}
\end{figure} 

%
%

\begin{figure}[h!]
\begin{center}
\subfigure[Phantom]{\includegraphics[width=0.45\textwidth]{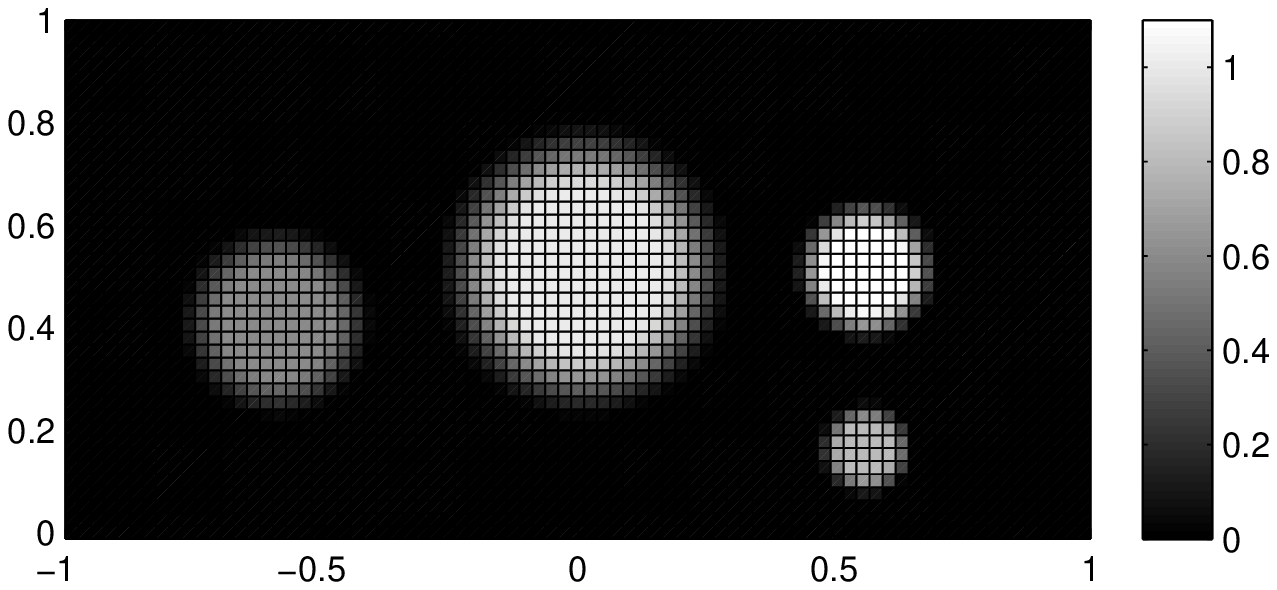}}    \\
\subfigure[Cross-section of the reconstruction in the $x_2x_3$-plane from exact data]{\includegraphics[width=0.45\textwidth]{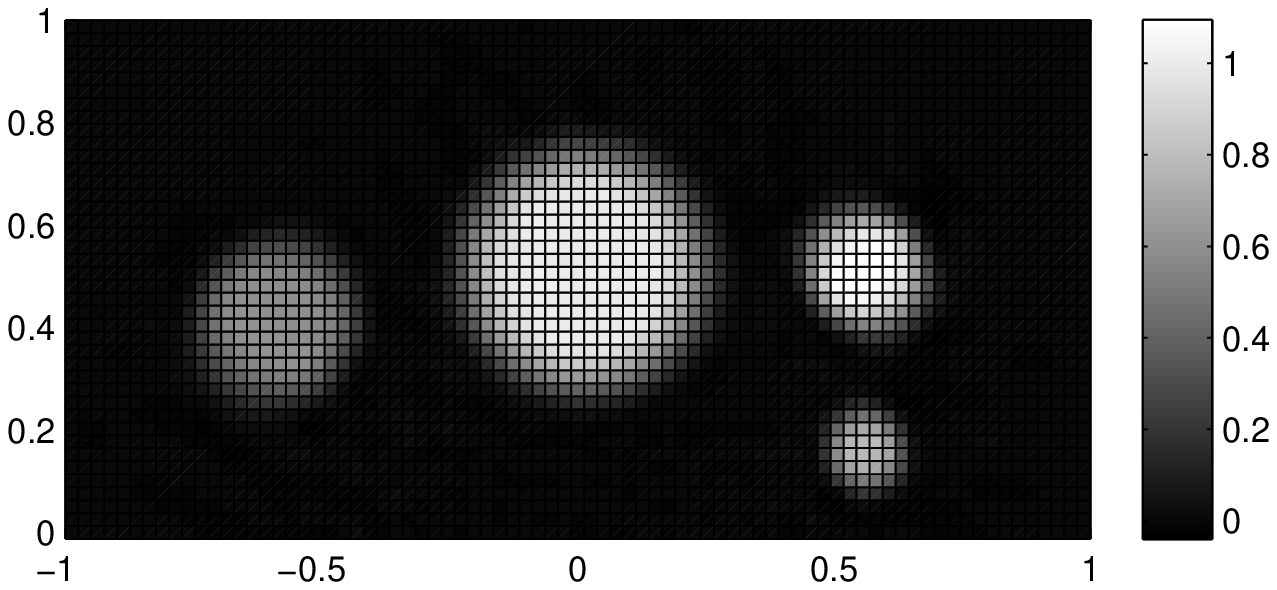}}
\subfigure[Cross-section of the reconstruction in the $x_2x_3$-plane from data with 20\% uniform random noise]{\includegraphics[width=0.45\textwidth]{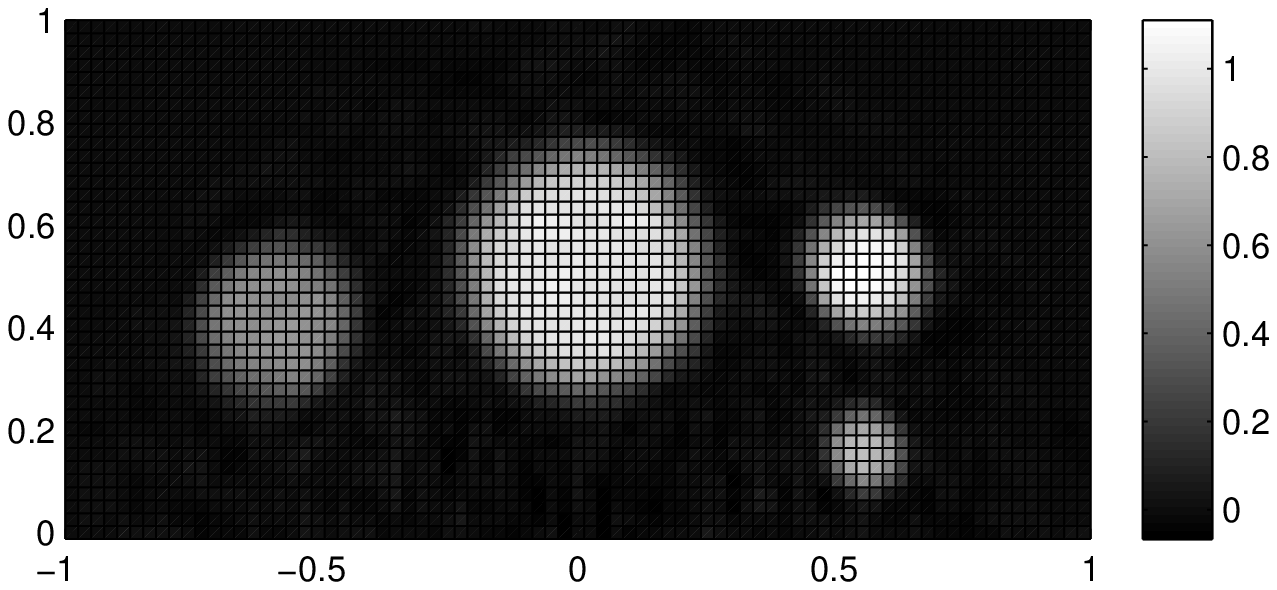}}\\
\subfigure[Slice of the reconstruction and phantom along the $x_3$-axis from exact data]{\includegraphics[width=0.45\textwidth]{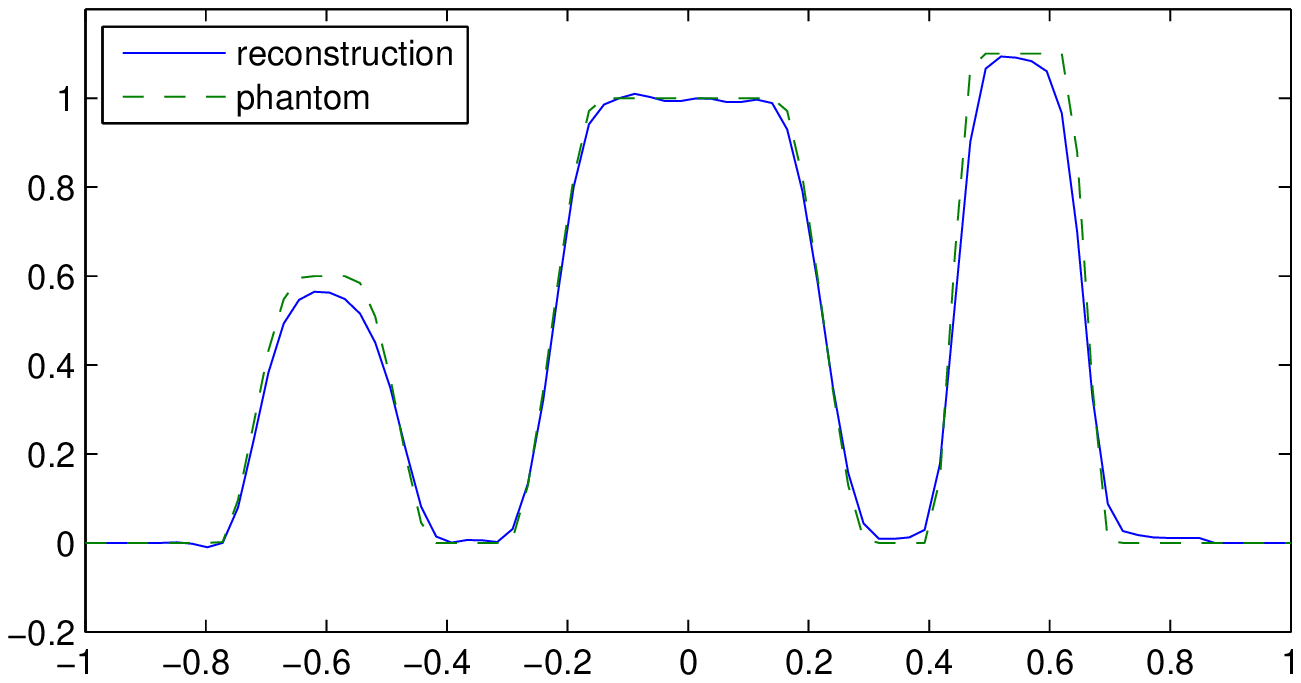}}
\subfigure[Slice of the reconstruction and phantom along the $x_3$-axis from data with 20\% uniform random noise]{\includegraphics[width=0.45\textwidth]{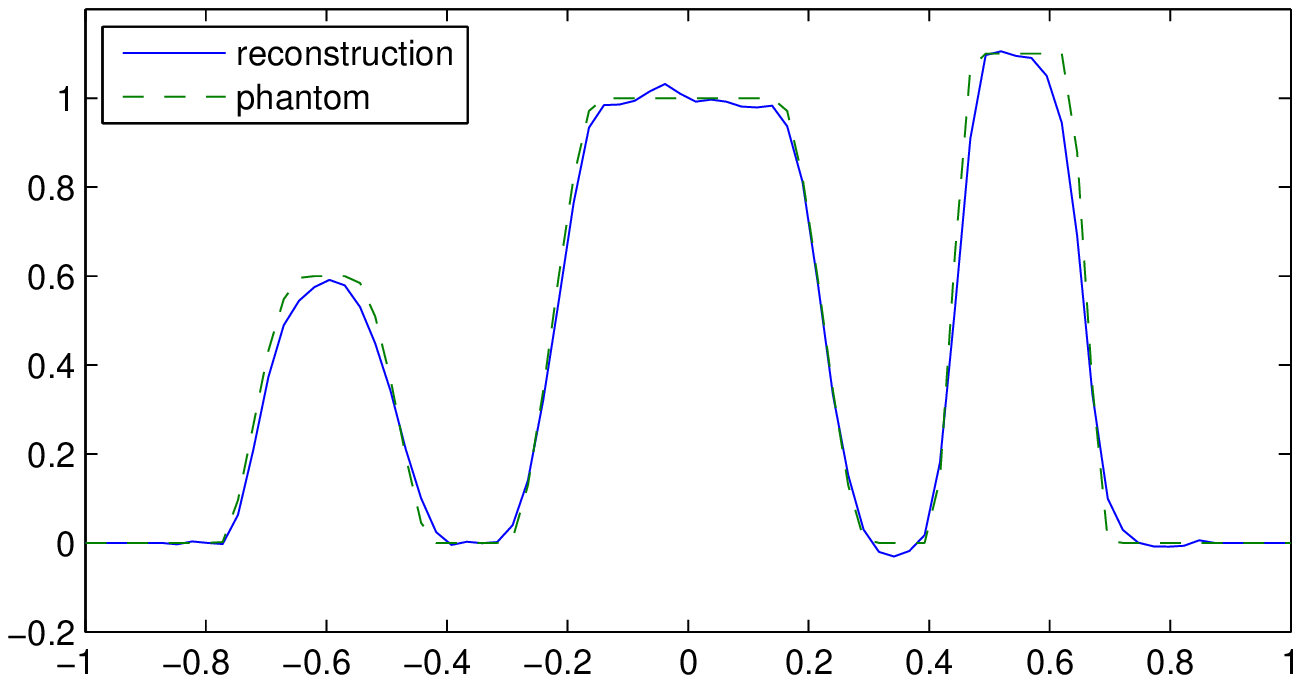}}\\
\subfigure[Slice of the reconstruction and phantom along the line $x_1=0,\;x_3=0.5$ from exact data]{\includegraphics[width=0.45\textwidth]{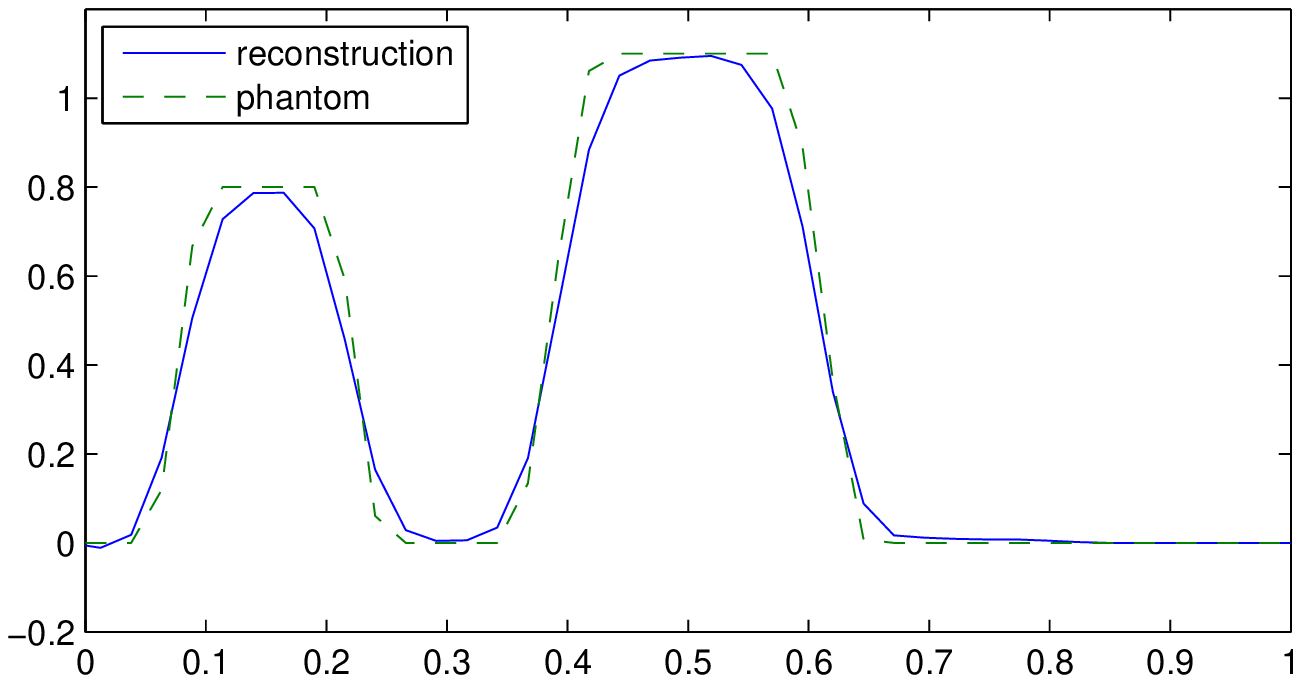}}
\subfigure[Slice of the reconstruction and phantom along the line $x_1=0,\;x_3=0.5$ from data with 20\% uniform random noise]{\includegraphics[width=0.45\textwidth]{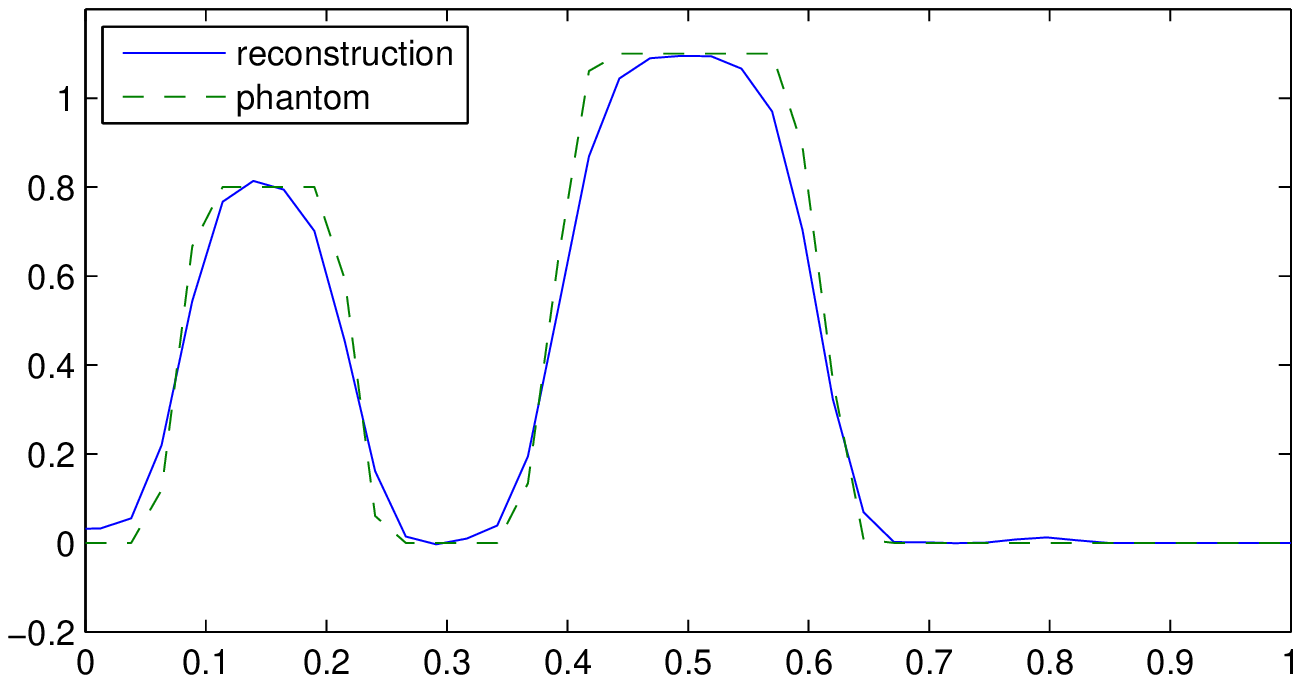}}
\end{center}

\caption{Reconstructions of a smooth phantom with centers $(0,-0.6,0.4)$, $(0,0,0.5)$, $(0,0.55,0.5)$,
and $(0,0.55,0.15)$, outer radii $0.2$, $0.3$, $0.15$, and $0.1$, inner radii $0.1$, $0.15$, $0.075$, and $0.05$, and central values 0.6, 1, 1.1, and 0.8, respectively}
\label{fig:4smoothballs}
\end{figure}

The Funk-Minkowski transform is inverted at each time step by using the inverse two-dimensional Radon transform
as prescribed by Theorem \ref{thm:funkradon}.  A direct inversion of $F$ using Theorem \ref{thm:funkradon}
would involve dividing by $\alpha_3^2$, which could cause instability near the $\alpha_1\alpha_2$-plane.
To avoid this, one possible remedy is to
cyclically permute $\theta_1$, $\theta_2$, and $\theta_3$ in (\ref{eq:FunkRadon})
and sum the results to obtain
\begin{equation*}
 2p(\uualpha,t)(\alpha_1^2+\alpha_2^2+\alpha_3^2)=2p(\uualpha,t).
\end{equation*}
In fact, we can further reduce the contribution from points near the $\alpha_i=0$ plane
(where $\alpha_i$ is the coordinate in the denominator of (\ref{eq:FunkRadon})
after a cyclic permutation of $\theta_1$, $\theta_2$, and $\theta_3$).
To achieve this we multiply the result of (\ref{eq:FunkRadon}) by $\alpha_i^k$ and 
sum the results:
\begin{equation*}
 2p(\uualpha,t)(\alpha_1^{k+2}+\alpha_2^{k+2}+\alpha_3^{k+2}).
\end{equation*}
The function $p(\uualpha,t)$ is then easily obtained by dividing by
$2(\alpha_1^{k+2}+\alpha_2^{k+2}+\alpha_3^{k+2})$.
In our reconstructions, we use $k=2$ in this formula.
To get the initial pressure field from $p(\uualpha,t)$, we use (\ref{eq:inversionofspherical}).


In our experiments, the function $p(\uutheta,t)$ is reconstructed on a spherical grid of 50 polar angles
and 200 azimuthal angles, for a grid of 50 values of $t$ on the interval $[0,2]$.
The polar angles range from $[\pi/25,\pi/2]$.
We start our polar grid slightly away from 0 because when we invert the Funk-Minkowski transform using Theorem \ref{thm:funkradon}, small values of the polar angle correspond to $s$ values near infinity, and having to account for large values of $s$ makes implementation of the inverse Radon transform difficult.
It is also not necessary to take polar angles
past $\pi/2$ because we are reconstructing the even part of $f$.
Our final reconstructions are calculated on an $80\times80\times80$ grid in $\vx$.

Figure \ref{fig:4sharpballs} shows the reconstruction of a phantom comprising four sharp balls of the form
(\ref{eq:sharpphantom}) from exact data.  A reconstruction of a smooth phantom comprising four functions of the form
(\ref{eq:smoothphantom}) from exact and noisy data is shown in Figure \ref{fig:4smoothballs}.

%
%

\section{Conclusion}\label{S:conclusions}
We studied a Radon-type transform arising in PAT with circular detectors and showed that it can be decomposed into the Minkowski-Funk transform and the spherical Radon transform. An inversion formula and range conditions were provided. In addition, numerical experiments demonstrated the feasibility of the reconstruction formula and showed good reconstructions from noisy data.
\section*{Acknowledgement}
S. Moon was supported by Basic Science Research Program through the National Research
Foundation of Korea (NRF) funded by the Ministry of Education, Science and Technology (NRF-
2012R1A1A1015116).



\nocite{kuchment12,gelfandgg03,palamodov98}
\bibliographystyle{plain}


\begin{thebibliography}{10}
\bibitem{agranovskyfk09}
M.~Agranovsky, D.~Finch, and P~Kuchment.
\newblock Range conditions for a spherical mean transform.
\newblock {\em Inverse Problems and Imaging}, 3(3):373--382, 2009.

\bibitem{agranovskykq07}
M.~Agranovsky, P.~Kuchment, and E.~T. Quinto.
\newblock Range descriptions for the spherical mean {Radon} transform.
\newblock {\em Journal of Functional Analysis}, 248(2):344 -- 386, 2007.

\bibitem{agranovskyn10}
M.~Agranovsky and L.~V Nguyen.
\newblock Range conditions for a spherical mean transform and global
  extendibility of solutions of the darboux equation.
\newblock {\em Journal d'Analyse Math{\'e}matique}, 112(1):351--367, 2010.


\bibitem{bell80}
A.G. Bell.
\newblock On the production and reproduction of sound by light.
\newblock {\em American Journal of Science}, 20:305--324, October 1880.

\bibitem{burgholzerbmghp07}
P.~Burgholzer, J.~Bauer-Marschallinger, H.~Gr{\"{u}}n, M.~Haltmeier, and
  G.~Paltauf.
\newblock Temporal back-projection algorithms for photoacoustic tomography with
  integrating line detectors.
\newblock {\em Inverse Problems}, 23(6):S62--S80, 2007.

\bibitem{daix13}
F.~Dai and Y.~Xu.
\newblock Spherical harmonics.
\newblock In {\em Approximation Theory and Harmonic Analysis on Spheres and
  Balls}, pages 1--27. Springer, 2013.

\bibitem{EvansBook}
L.~Evans. 
\newblock  Partial differential equations: Graduate studies in Mathematics.
 \newblock American Mathematical Society, v. 2, 1998.

\bibitem{finchpr04}
D.~Finch, S.~Patch, and Rakesh.
\newblock Determining a function from its mean values over a family of spheres.
\newblock {\em SIAM Journal on Mathematical Analysis}, 35(5):1213--1240, 2004.

\bibitem{finchr09}
D.~Finch and Rakesh.
\newblock Recovering a function from its spherical mean values in two and three dimensions.
In \cite{wang09}, pages 77--88.

\bibitem{gardner95}
R.J. Gardner.
\newblock {\em Geometric Tomography}.
\newblock Encyclopedia of mathematics and its applications. Cambridge
  University Press, 1995.

\bibitem{gelfandgg03}
I.M. Gelʹfand, S.G. Gindikin, and M.I. Graev.
\newblock {\em Selected Topics in Integral Geometry}.
\newblock Translations of Mathematical Monographs. American Mathematical
  Society, 2003.

\bibitem{gindikinrs93}
S.~Gindikin, J.~Reeds, and L.~Shepp.
\newblock Spherical tomography and spherical integral geometry.
\newblock In E.T. Quinto, M.~Cheney, P.~Kuchment, and American~Mathematical
  Society, editors, {\em Tomography, Impedance Imaging, and Integral Geometry:
  1993 AMS-SIAM Summer Seminar on the Mathematics of Tomography, Impedance
  Imaging, and Integral Geometry, June 7-18, 1993, Mount Holyoke College,
  Massachusetts}, Lectures in Applied Mathematics Series, pages 83--92.
  American Mathematical Society, 1994.

\bibitem{grattpnp11}
S.~Gratt, K.~Passler, R.~Nuster, and G.~Paltauf.
\newblock Photoacoustic section imaging with an integrating cylindrical
  detector.
\newblock {\em Biomedical Optics Express}, 2(11):2973--2981, November 2011.

\bibitem{groemer96}
H.~Groemer.
\newblock {\em Geometric Applications of Fourier Series and Spherical
  Harmonics}.
\newblock Encyclopedia of Mathematics and its Applications. Cambridge
  University Press, 1996.

\bibitem{haltmeier09}
M.~Haltmeier.
\newblock Frequency domain reconstruction for photo-and thermoacoustic
  tomography with line detectors.
\newblock {\em Mathematical Models and Methods in Applied Sciences},
  19(02):283--306, 2009.

\bibitem{haltmeiersbp04}
M.~Haltmeier, O.~Scherzer, P.~Burgholzer, and G.~Paltauf.
\newblock Thermoacoustic computed tomography with large planar receivers.
\newblock {\em Inverse Problems}, 20(5):1663--1673, 2004-10-01T00:00:00.

\bibitem{haltmeierss05}
M.~Haltmeier, T.~Schuster, and O.~Scherzer.
\newblock Filtered backprojection for thermoacoustic computed tomography in spherical geometry.
\newblock {\em Mathematical Models and Methods in Applied Sciences},
  28:1919--1937, 2005.


\bibitem{helgason99radon}
S.~Helgason.
\newblock {\em The {Radon} Transform}.
\newblock Progress in Mathematics. Birkh{\"a}user, Boston, 1999.

\bibitem{helgason11}
S. Helgason.
\newblock {\em Integral geometry and Radon transforms}.
\newblock Springer, 2011.


\bibitem{kuchment12}
P.~Kuchment.
\newblock Mathematics of hybrid imaging: A brief review.
\newblock In Irene Sabadini and Daniele~C Struppa, editors, {\em The
  Mathematical Legacy of Leon Ehrenpreis}, volume~16 of {\em Springer
  Proceedings in Mathematics}, pages 183--208. Springer Milan, 2012.

\bibitem{kuchmentk08}
P.~Kuchment and L.~Kunyansky.
\newblock Mathematics of thermoacoustic tomography.
\newblock {\em European Journal of Applied Mathematics}, 19:191--224, 2008.

\bibitem{kunyansky07}
L.A. Kunyansky.
\newblock Explicit inversion formulae for the spherical mean {Radon} transform.
\newblock {\em Inverse Problems}, 23(1):373, 2007.

\bibitem{kunyansky071}
L.A. Kunyansky.
\newblock A series solution and a fast algorithm for the inversion of the
  spherical mean {Radon} transform.
\newblock {\em Inverse Problems}, 23(6):S11, 2007.

\bibitem{smoon}
S.~Moon.
\newblock {\em Properties of some integral transforms arising in tomography}.
\newblock Dissertation, Texas A\&M University, 2013.


\bibitem{mooncrt13}
S.~{Moon}.
\newblock {On the determination of a function from cylindrical {Radon}
  transforms}.
\newblock {\em ArXiv e-prints:1309.4143}, September 2013.

\bibitem{moontat14}
S.~{Moon}.
\newblock {A Radon-type transform arising in Photoacoustic Tomography with
  circular detectors}.
\newblock {\em to appear in Journal of Inverse and Ill-Posed Problems}.

\bibitem{natterer01}
F.~Natterer.
\newblock {\em The Mathematics of Computerized Tomography}.
\newblock Classics in Applied Mathematics. Society for Industrial and Applied
  Mathematics, Philadelphia, 2001.

\bibitem{nattererw01}
F.~Natterer and F.~W{\"u}bbeling.
\newblock {\em Mathematical methods in image reconstruction}.
\newblock SIAM Monographs on mathematical modeling and computation. SIAM,
  Society of industrial and applied mathematics, Philadelphia (Pa.), 2001.
  
\bibitem{Nguyen_FamilyInversionTAT}
L.V.~Nguyen.
\newblock {\em A family of inversion formulas in thermoacoustic tomography}.
\newblock {\em Inverse Problems and Imaging}, 3(4):649--675, 2009.

\bibitem{palamodov98}
V.P. Palamodov.
\newblock Reconstruction from line integrals in spaces of constant curvature.
\newblock {\em Mathematische Nachrichten}, 196(1):167--188, 1998.

\bibitem{seeley66}
R.T. Seeley.
\newblock Spherical harmonics.
\newblock {\em The American Mathematical Monthly.}, 73(4):115--121, April 1966.

\bibitem{titchmarsh88}
E.C. Titchmarsh.
\newblock {\em Introduction to the Theory of Fourier Integrals}.
\newblock Repro Pfeffer, 1988.

\bibitem{wang09}
L.~Wang (ed).
\newblock Photoacoustic imaging and spectroscopy.  CRC Press, Boca Raton.

\bibitem{xuw06}
Y.~Xu and L.V. Wang.
\newblock Photoacoustic imaging in biomedicine.
\newblock {\em Review of Scientific Instruments}, 77(4):041101--041122, April
  2006.

\bibitem{zangerls10}
G.~Zangerl and O.~Scherzer.
\newblock Exact reconstruction in photoacoustic tomography with circular
  integrating detectors {II}: Spherical geometry.
\newblock {\em Mathematical Methods in the Applied Sciences},
  33(15):1771--1782, 2010.

\bibitem{zangerlsh09}
G.~Zangerl, O.~Scherzer, and M.~Haltmeier.
\newblock Circular integrating detectors in photo and thermoacoustic
  tomography.
\newblock {\em Inverse Problems in Science and Engineering}, 17(1):133--142,
  2009.

\bibitem{zangerls09}
G.~Zangerl, O.~Scherzer, and M.~Haltmeier.
\newblock Exact series reconstruction in photoacoustic tomography with circular
  integrating detectors.
\newblock {\em Communications in Mathematical Sciences}, 7(3):665--678, 2009.

\end{thebibliography}

 \appendices
\section{Proofs of Theorem \ref{thm:funkradon} and Lemmas \ref{lem:derivative} and \ref{lem:integral}}
\newtheorem*{thm:funkradon}{Theorem \ref{thm:funkradon}}
\begin{thm:funkradon}
Let $\phi$ be a continuous and even function on $S^2$, i.e., $\phi(\uutheta)=\phi(-\uutheta)$.
If 
$$
F\phi(\uutheta)=\intL_{\uutheta\cdot\uualpha=0}\phi(\uualpha)dS(\uualpha), \qquad\mbox{ for  } \uutheta\in S^2,
$$
then, if $\uutheta\neq(0,0,1)$ and $\uutheta\neq(0,0,-1)$, the following relation holds,
\begin{equation}\label{eq:relationbetweenradon}
F\phi(\uutheta)=\frac{1}{|\theta'|}R\Phi\left(\frac{\theta'}{|\theta'|},-\frac{\theta_3}{|\theta'|}\right),
\end{equation}
where $\uutheta=(\theta',\theta_3)=(\theta_1,\theta_2,\theta_3)\in S^2, \uualpha = (\alpha',\alpha_3)\in S^2$, $\Phi(\alpha'/\alpha_3)=2\phi(\uualpha)\alpha_3^2$, and  
$$
R[\Phi(x_1,x_2)](\uuomega,s)=\intR \Phi(s\uuomega+\nu \uuomega^\perp)d\nu,\quad\mbox{ for  } (\uuomega,s)\in S^1\times\RR.
$$
\end{thm:funkradon}
\begin{proof}
By definition, $F\phi$ can be written as
$$
F\phi(\uutheta)=\intL^{2\pi}_0\phi(\hat u\cos t+\hat v \sin t)dt,
$$
where 
$$
\hat u=(u_1,u_2,u_3) = \frac{1}{|\theta'|}(-\theta_2,\theta_1,0)\in S^2\;\mbox{and}\; \hat v =(v_1,v_1,v_3)=\frac{1}{|\theta'|}(\theta_1\theta_3,\theta_2\theta_3,-|\theta'|^2)\in S^2.$$
Note that $\hat u,\hat v$ and $\uutheta$ are orthonormal, so that $\hat{u} \cos t + \hat{v}\sin t, $ for $t\in[0,2\pi]$ is a parametrization of the great circle orthogonal to $\uutheta$.

By the evenness of $\phi$ and the definition of $\Phi$, we have
\begin{equation*}
\begin{split}
F\phi(\uutheta)&=\intL^{\pi}_0\Phi\left(\frac{u'\cos t+v' \sin t}{v_3\sin t}\right)\frac{dt}{v_3^2\sin^2 t}=\intL^{\pi}_0\Phi\left(\frac{u'}{v_3}\cot t+\frac{v' }{v_3}\right)\frac{dt}{v_3^2\sin^2 t}\\
&=\intR\Phi\left(\frac{u'}{v_3} t+\frac{v' }{v_3}\right)v_3^{-2} dt,
\end{split}
\end{equation*}
where in the last line, we changed the variables $\cot t\to t$.
Here $u'=(-\theta_2,\theta_1)/|\theta'|$ and $v'=\theta'/\sqrt{|\theta'|^2+|\theta'|^4/\theta_3^2}$.
Thus $F\phi(\uutheta)$ is the integral of $\Phi$ over the line with direction $u'/v_3$ and passing through $v'/v_3$.
We can represent $F\phi$ as
\begin{equation*}
\begin{split}
F\phi(\uutheta)&=\intR\Phi\left(u'\frac{t}{v_3} +\frac{v' }{|v' |}\frac{|v' |}{v_3}\right)v_3^{-2} dt=\intR\Phi\left(u't +\frac{v' }{|v'|}\frac{|v' |}{v_3}\right)|v_3|^{-1} dt\\
&= |v_3|^{-1} R\Phi\left(\frac{v' }{|v' |},\frac{|v' |}{v_3}\right),
\end{split}
\end{equation*}
where in the second line, we changed the variables $t/v_3\to t$.
The following identities complete our proof:
$$
\frac{v' }{|v' |}=\frac{\theta'}{|\theta'|},\quad \frac{|v' |}{v_3}=-\frac{\theta_3}{|\theta'|},\mbox{ and } |v_3|^{-1}=\frac1{|\theta'|}.
$$
\end{proof} 


\newtheorem*{lem:derivative}{Lemma \ref{lem:derivative}}
\begin{lem:derivative}
Let $\phi\in C^\infty([0,\epsilon])$ for some $\epsilon>0$ satisfy $\phi(0)=0$, and let $n$ be a positive integer.
For any positive integer $k\leq n$, we have
$$
\intL^t_0\phi(\tau) \partial_t^{k}\left(\frac{(t-\tau)^{n+1}}{t}\right)d\tau\to0\qquad\mbox{as}\quad t\to 0.
$$
\end{lem:derivative}

\begin{proof}
It is enough to show that for any positive integer $m$,
\begin{equation}\label{eq:lemlimit}
\intL^t_0\phi(\tau)\frac{(t-\tau)^{m'}}{t^m}d\tau\to0\qquad\mbox{as}\quad t\to 0,\qquad\mbox{for any} \quad m'\geq m
\end{equation}
since 
$$
\partial_t^{k}\left(\frac{(t-\tau)^{n+1}}{t}\right)=\sum^k_{j=0} c_{k,j}(t-\tau)^{n+1-j}t^{-k+j-1},
$$
where $c_{k,j}$ are some constants depending on $k$ and $j$.
To show \eqref{eq:lemlimit}, it is enough that for any positive integer $m$,
\begin{equation}\label{eq:lemlimit1}
\intL^t_0\phi(\tau)\frac{\tau^{m'}}{t^m}d\tau\to0\qquad\mbox{as}\quad t\to 0,\qquad\mbox{for any} \quad m'\geq m
\end{equation}
since 
$$
(t-\tau)^{m'}=\sum^{m'}_{j=0} \;_{m'}C_jt^j\tau^{m'-j},
$$
where $_{m'}C_j$ is the number of $j$-combination from a given set of $m'$ elements.
Applying L'H\^opital's rule to \eqref{eq:lemlimit1}, we have
$$
\frac{\phi(t)t^{m'}}{mt^{m-1}}=\phi(t)t^{m'-m+1}\to 0 \qquad\mbox{as}\quad t\to 0.
$$
\end{proof}

\newtheorem*{lem:integral}{Lemma \ref{lem:integral}}
\begin{lem:integral}
Let $\phi\in C^\infty([0,\epsilon])$ satisfy $\phi(0)=0$, and let $h(t)=\int_0^t\phi(\tau)d\tau/t$.
Then $h\in C^\infty([0,\epsilon])$ and $h(0)=0$.
\end{lem:integral}

\begin{proof}
The smoothness of $h$ is obvious except at $t=0$.
In order to prove smoothness at $t=0$, we will use mathematical induction.
First of all, $h(0)$ is equal to $\phi(0)$, which is equal to zero, by L'H\^opital's rule.
Suppose $h^{(n-1)}(0)=\phi^{(n-1)}(0)/n$ and $h^{(n-1)}(t)$ is continuous at 0.
Consider the Taylor expansion of $\int^t_0\phi(\tau)d\tau$ at $t=0$.
Then for any $n$, we have
\begin{equation*}
\begin{split}
\intL^t_0\phi(\tau)d\tau=&t\phi(0)+\frac {t^2}{2!}\phi'(0)+\frac {t^3}{3!}\phi^{(2)}(0)+\cdots \\
&+\frac {t^n}{n!}\phi^{(n-1)}(0)+\frac {t^{(n+1)}}{(n+1)!}\phi^{(n)}(0)+\intL^t_0\frac {\phi^{(n+1)}(\tau)}{(n+1)!}(t-\tau)^{n+1}d\tau,
\end{split}
\end{equation*}
where $\phi^{(n)}(0)$ is the $n$-th derivative of $\phi$ evaluated at $t=0$. 
Hence, we have
\begin{equation*}
\begin{split}
h (t)=&\frac1 t\intL^t_0\phi(\tau)d\tau\\
=&\phi(0)+\frac {t}{2!}\phi'(0)+\cdots+\frac {t^{n}}{(n+1)!}\phi^{(n)}(0)+\frac1t\intL^t_0\frac {\phi^{(n+1)}(\tau)}{(n+1)!}(t-\tau)^{n+1}d\tau.
\end{split}
\end{equation*}
For any $k\leq n$, the $k$-derivative of the last term (or the reminder term) is
\begin{equation}\label{eq:derivativeofreminder}
\begin{split}
\partial_t^k\intL^t_0\frac {\phi^{(n+1)}(\tau)}{(n+1)!}\frac{(t-\tau)^{n+1}}td\tau&=\partial_t^{k-1}\intL^t_0\frac {\phi^{(n+1)}(\tau)}{(n+1)!}\partial_t\left(\frac{(t-\tau)^{n+1}}t\right)d\tau\\
&\qquad\vdots\\
&=\intL^t_0\frac {\phi^{(n+1)}(\tau)}{(n+1)!}\partial_t^k\left(\frac{(t-\tau)^{n+1}}t\right)d\tau,
\end{split}
\end{equation}
since $\partial_t^{k'}((t-\tau)^{n+1}/t)=(t-\tau)^{n+1-k'}[(n+1)!/(n-k')!t^{-1}+\cdots+(-1)^{k'}k'!(t-\tau)^{k'}t^{-k'-1}]$ for $k'\leq n$. 
Hence, for $t$ near 0, we have
$$
h^{(n-1)}(t)=\frac 1{n}\phi^{(n-1)}(0)+\frac t{n+1}\phi^{(n)}(0)+\intL^t_0\frac {\phi^{(n+1)}(\tau)}{(n+1)!}\partial_t^{n-1}\left(\frac{(t-\tau)^{n+1}}t\right)d\tau
$$
and 
$$
h^{(n)}(t)=\frac 1{n+1}\phi^{(n)}(0)+\intL^t_0\frac {\phi^{(n+1)}(\tau)}{(n+1)!}\partial_t^{n}\left(\frac{(t-\tau)^{n+1}}t\right)d\tau.
$$
By Lemma~\ref{lem:derivative}, we have
$$
h^{(n)}(t)\to\frac 1{n+1}\phi^{(n)}(0)\qquad\mbox{as}\qquad t\to0.
$$
Hence we have
\begin{equation*}
\begin{split}
h^{(n)}(0)&=\displaystyle\lim_{t\to0}\frac1t\left(\frac t{n+1}\phi^{(n)}(0)+\intL^t_0\frac {\phi^{(n+1)}(\tau)}{(n+1)!}\partial_t^{n-1}\left(\frac{(t-\tau)^{n+1}}t\right)d\tau\right)\\
&=\displaystyle\frac {\phi^{(n)}(0)}{n+1}+\lim_{t\to0}\intL^t_0\frac {\phi^{(n+1)}(\tau)}{(n+1)!}\partial_t^{n}\left(\frac{(t-\tau)^{n+1}}t\right)d\tau=\frac{\phi^{(n)}(0)}{n+1},
\end{split}
\end{equation*}
where in the second line, we used L'H\^opital's rule, \eqref{eq:derivativeofreminder}, and Lemma \ref{lem:derivative}.
This means that $h$ is $n$-th differentiable at 0 and $h^{(n)}(t)$ is continuous at 0.
\end{proof}
\section{Verification of \eqref{eq:FBP}}
It is natural to apply an inversion formula for the regular Radon transform to obtain the inversion formula for the Minkowski-Funk transform when the relation \eqref{eq:relationbetweenradon} is given.
However, the filtered backprojection formula\footnote{In \cite[Thoerem 3.1 in Chapter 1]{helgason11}, the function $\Phi$ can be recovered from the Radon transform by the following inversion formula
$$
\Phi(x_1,x_2)=\frac{1}{4\pi}\intRR|(x_1,x_2)-(y_1,y_2)|^{-1}R^\# R\Phi(y_1,y_2)dy_1dy_2
$$
provided $\Phi(x_1,x_2)=O(|(x_1,x_2)|^{-N})$ for some $N>1$.
However, the filtered backprojection formula we want to apply is $\Phi(x_1,x_2)=(4\pi)^{-1}R^\# HR\Phi$.
}
 is typically stated for functions in the Schwartz space (see \cite{helgason99radon,helgason11,natterer01}).
For completeness, we verify here that the filtered backprojection formula can still be applied to (\ref{eq:FunkRadon2}).

From now on, we omit $t$, an unnecessary variable here.
From $\Phi(\alpha'/\alpha_3)=2\phi(\uualpha)\alpha_3^2$ for $\uualpha=(\alpha',\alpha_3)\in S^2$, we notice that
$$
\Phi(x_1,x_2)=\phi\left(\frac{(x_1,x_2,1)}{\sqrt{1+x^2_1+x_2^2}}\right)\frac1{1+x^2_1+x_2^2}
$$
 is a smooth and $L^2$-function, i.e.,
$$
\intL_{\RR^2}|\Phi(x_1,x_2)|^2dx_1dx_2<\infty.
$$
By the projection slice theorem, we have $\hat \Phi(\sigma\uuomega)=\widehat{R\Phi}(\uuomega,\sigma)$ for a Schwartz function $\Phi$, where
$\hat \Phi$ is the 2D Fourier transform of $\Phi$ and $\widehat{R\Phi}$ is the 1D Fourier transform of $R\Phi$.
By an approximation argument, the projection slice theorem still holds for an $L^2$-function $\Phi$. 
Since the Fourier transform is a unitary isomorphism on the $L^2$-space, the set of $L^2$-functions, we have
\begin{equation}\label{eq:inversionofradon}
\begin{array}{ll}
\Phi(x_1,x_2)&\displaystyle=\frac{1}{(2\pi)^2}\lim_{R\to\infty}\intL_{|(\xi_1,\xi_2)|<R}\widehat{R\Phi}((\xi_1,\xi_2)/|(\xi_1,\xi_2)|,|(\xi_1,\xi_2)|)e^{i(x_1,x_2)\cdot(\xi_1,\xi_2)}d\xi_1d\xi_2\\
&\displaystyle=\frac{1}{(2\pi)^2}\lim_{R\to\infty}\intL_{S^1}\intL_{\sigma<R}\widehat{R\Phi}(\uuomega,\sigma)|\sigma| e^{i\sigma(x_1,x_2)\cdot\uuomega}d\sigma d\uuomega\\
&\displaystyle=\frac{1}{(2\pi)^2}\lim_{R\to\infty}\intL_{S^1}\intL_{\sigma<R}i\operatorname{sgn}(\sigma)\widehat{\partial_sR\Phi}(\uuomega,\sigma)e^{i\sigma(x_1,x_2)\cdot\uuomega} d\sigma d\uuomega,
\end{array}
\end{equation}
where in the second line, we changed the variables $(\xi_1,\xi_2)\to(\sigma,\uuomega)\in[0,\infty)\times S^1$.
Since $F\phi$ is smooth and bounded, $R\Phi(\uuomega,s)$ is a smooth and $L^2$-function and so is $\partial_sR\Phi$.
We notice that Hilbert transform $Hu$ is also a unitary isomorphism on the $L^2$-space (by the Titchmarsh theorem~\cite[Chapter 5]{titchmarsh88}).
Therefore, \eqref{eq:inversionofradon} becomes (\ref{eq:FBP}).

\end{document}